\documentclass[11pt]{article}

\usepackage{amsmath}
\usepackage{amssymb}
\usepackage{amsthm}
\usepackage{enumitem}
\usepackage[text={0.75\paperwidth,0.75\paperheight}]{geometry} %for page layout
\usepackage{graphicx} %used for rotating symbols (depends on the *.dvi viewer, but appears in *.pdf)
\usepackage[colorlinks=true]{hyperref}
\usepackage{mathtools}
\usepackage{titlesec}
\usepackage[all]{xy}
\DeclareMathAlphabet{\mathpzc}{OT1}{pzc}{m}{it}

\usepackage[parfill]{parskip}
\begingroup
\makeatletter
   \@for\theoremstyle:=definition,remark,plain\do{%
     \expandafter\g@addto@macro\csname th@\theoremstyle\endcsname{%
        \addtolength\thm@preskip\parskip
     }%
   }
\endgroup

\xyoption{tips}
\SelectTips{cm}{10}
\UseTips

\swapnumbers \theoremstyle{plain}
\newtheorem{lem}[subsubsection]{Lemma}
\newtheorem{prop}[subsubsection]{Proposition}
\newtheorem{thm}[subsubsection]{Theorem}
\newtheorem{mainthm}[subsection]{Theorem}
\newtheorem*{thm*}{Theorem}

\theoremstyle{definition}

\newtheorem{ex}[subsubsection]{Example}
\newtheorem{exs}[subsubsection]{Examples}
\newtheorem{rem}[subsubsection]{Remark}
\newtheorem{rems}[subsubsection]{Remarks}

\titleformat{\subsection}[runin]{\normalfont\normalsize\bfseries}{\thesubsection}{0.4em}{}[.]
\titleformat{\subsubsection}[runin]{\normalfont\normalsize\bfseries}{\thesubsubsection}{0.4em}{}[.]
\titleformat{\paragraph}[runin] {\normalfont\normalsize\bfseries}{\theparagraph}{0.4em}{}[.]
\titleformat{\subparagraph}[runin] {\normalfont\normalsize\bfseries}{\thesubparagraph}{0.4em}{}[.]

\titlespacing*{\subsection} {0pt}{3.25ex plus 1ex minus .2ex}{0.4em} 
\titlespacing*{\subsubsection}{0pt}{3.25ex plus 1ex minus .2ex}{0.4em} 
\titlespacing*{\paragraph} {0pt}{3.25ex plus 1ex minus .2ex}{0.4em} 
\titlespacing*{\subparagraph} {\parindent}{3.25ex plus 1ex minus .2ex}{0.4em}

\newenvironment{examples}[1][\mbox{}]
{\goodbreak\begin{exs}#1\vspace{-.3em}\begin{enumerate}[label=(\arabic{*})]}{\end{enumerate}\end{exs}\par\addvspace{\baselineskip}}

\newcommand{\into}{\hookrightarrow} %embedding arrow
\renewcommand{\to}{\xrightarrow{\rule{1.45ex}{0ex}}} %to arrow

\newcommand{\xtimes}{\boxtimes}

\newcommand{\dand}{\qquad\text{and}\qquad} %"display" \and command
\newcommand{\C}{\mathsf{c}} %exponent for set-complement
\newcommand{\comma}{\mathrel{\downarrow}}
\newcommand{\conv}{\mathrm{conv}}
\newcommand{\df}{\emph} %used when defining a term
\renewcommand{\epsilon}{\varepsilon}

\newcommand{\inv}{{-1}} %exponent for inverse
\newcommand{\kleisli}{\circ} %kleisli composition
\newcommand{\nbhd}{\mathrm{nbhd}}
\newcommand{\ob}{\mathop{\mathrm{ob}}\nolimits} %objects
\newcommand{\op}{\mathrm{op}} %exponent for opposite category
\newcommand{\dc}{{\mathop{\downarrow}\nolimits}} %down-closure
\newcommand{\cat}[1]{\mathsf{#1}}

\newcommand{\Ccd}{\cat{Ccd}} %category of constructively completely distributive lattices
\newcommand{\Cnt}{\cat{Cnt}} %category of continuous lattices
\newcommand{\Int}{\cat{Int}} %category of interior spaces
\newcommand{\Mnd}{\cat{Mnd}} %category of monads
\newcommand{\PrOrd}{\cat{PrOrd}} %category of preordered sets
\newcommand{\Set}{\cat{Set}} %category of sets
\newcommand{\Sup}{\cat{Sup}} %category of sup-semilattices
\newcommand{\Top}{\cat{Top}} %category of topological spaces
\newcommand{\Cats}[1]{#1\text{-}\cat{Cat}} %category of #1-categories
\newcommand{\Mons}[1]{\cat{Mon(\Set_{#1})}} %category of #1-monoids
\DeclareMathOperator{\Up}{\mathrm{Up}\hspace*{0ex}} %Up functor on Ord
\newcommand{\mon}[1]{\mathbb{#1}}

\newcommand{\mF}{\mon{F}} %filter monad
\newcommand{\mP}{\mon{P}} %powerset monad
\newcommand{\mS}{\mon{S}} %generic monad
\newcommand{\mT}{\mon{T}} %generic monad
\newcommand{\mU}{\mon{U}} %up-set monad

\newcommand{\mfrak}[1]{\mathpzc{#1}}

\newcommand{\ff}{\mfrak{f}}
\newcommand{\fF}{\mfrak{F}}

\newcommand{\fx}{\mfrak{x}}
\newcommand{\fX}{\mfrak{X}}
\newcommand{\fy}{\mfrak{y}}

\newcommand{\cali}[1]{\mathcal{#1}}

\newcommand{\sA}{\cali{A}}

\newcommand{\sV}{\cali{V}}

\newcommand{\ev}{\mathrm{ev}} %evaluation
\newcommand{\two}{\mathsf{2}} %quantale

\begin{document}

\title{Exponential Kleisli monoids as Eilenberg--Moore algebras}

\author{Dirk Hofmann%\\
%Departamento de Matem\'atica\\
%Universidade de Aveiro\\
%3810-193 Aveiro, Portugal\\
%\small\texttt{dirk@ua.pt}
\and Fr\'ed\'eric Mynard%\\
%Department of Mathematical Sciences\\
%Georgia Southern University\\
%Statesboro, GA 30460, United States\\
%\texttt{fmynard@georgiasouthern.edu} 
\and Gavin J. Seal\thanks{Financial support by a Marie-Curie International Reintegration Grant is gratefully acknowledged.}%\\
%Institute of Geometry, Algebra and Topology\\
%Ecole Polytechnique F\'ed\'erale de Lausanne\\
%1015 Switzerland\\
%\texttt{gavin.seal@epfl.ch}
}

%\date{July 22, 2009}

%\address{Institute of Geometry, Algebra and Topology\\Ecole Polytechnique F\'ed\'erale de Lausanne\\1015 Switzerland\\\texttt{gseal@fastmail.fm}}

%Mathematics Subject Classification: 18C20, 18B30, 54A05

\maketitle

\begin{abstract}
Lax monoidal powerset-enriched monads yield a monoidal structure on the category of monoids in the Kleisli category of a monad. Exponentiable objects in this category are identified as those Kleisli monoids with algebraic structure. This result generalizes the classical identification of exponentiable topological spaces as those whose lattice of open subsets forms a continuous lattice.
\end{abstract}


%%%%%%%%%%%%%%%%%%%%%%%%%%%%%%%%%%%%%%%%%%%%%%%%%%%%%%%%%%%%%%%%%%%%%%%%%%%%%%%%%%%%%%%%%%%%%%%%%%%%%%%%

%\section{Introduction}
%
%None here

%%%%%%%%%%%%%%%%%%%%%%%%%%%%%%%%%%%%%%%%%%%%%%%%%%%%%%%%%%%%%%%%%%%%%%%%%%%%%%%%%%%%%%%%%%%%%%%%%%%%%

\nocite{ColRic:01}

\section{Introduction}

The classical identification of exponentiable topological spaces by Day and Kelly (\cite{DayKel:70}, see \cite{EscLawSim:04} or \cite{Isb:86}) can be summarized as follows:

\begin{quote}
A topological space is exponentiable if and only if its set of open subsets, ordered by inclusion, forms a continuous lattice.
\end{quote}

In \cite{Day:75}, Day subsequently showed that continuous lattices were precisely algebras for the filter monad $\mF=(F,\mu,\eta)$ on $\Set$. The occurrence of a monad in the cited exponentiability result seemed anecdotal---until G\"ahler remarked in \cite{Gae:92} that a topological space is exactly a monoid in the Kleisli category of $\mF$ (see Section \ref{ssec:KlMon} for details). By identifying the set of open subsets of a topological space $X$ with the set $[X,\two]$ of continuous maps into the Sierpinski space $2$, and defining a map $\conv:F[X,2]\to [X,2]$ that sends a filter $\fF$ on the set of open subsets of $X$ to the set $\bigcup_{A\in\fF}(\,\bigcap A)^\circ$ (with $(-)^\circ$ denoting the interior operation), Day and Kelly's result can thus be rephrased as:

\begin{quote}
An $\mF$-monoid $X$ is exponentiable if and only if $([X,2],\conv)$ is an $\mF$-algebra.
\end{quote}

The categorical product, that appears in the definition of exponentiability as a functor
\[
X\times(-):\Top\to\Top\ ,
\]
can itself be related to $\mF$, via the monoidal natural transformation $\kappa:F(-)\times F(-)\to F(-\times -)$ whose components sends a pair of filters to their product. These features point to the ubiquitous role of the filter monad. In this article, we introduce sufficient conditions for a monad $\mT$ on $\Set$ to reproduce the behavior of $\mF$: such a monad will be called a lax monoidal powerset-enriched monad, and Day and Kelly's result becomes our Theorems \ref{thm:1} and \ref{thm:2},  summarized as follows.

\begin{mainthm}\label{thm:3}
Let $\mT$ be a lax monoidal powerset-enriched monad, and $\sV$ a family of $\mT$-algebras that is  initially dense in $\Mons{\mT}$.

A $\mT$-monoid $(X,\alpha)$ is exponentiable if and only if $([X,V],\conv)$ is a $\mT$-algebra for all $V\in\sV$.
\end{mainthm}

The definition of the category $\Mons{\mT}$ of $\mT$-monoids is given in \ref{ssec:KlMon}, and that of a lax monoidal powerset-enriched monad in \ref{ssec:P-EnrichedMnd} and \ref{ssec:MonoidalP-Mnd}. The fact that a $\mT$-algebra can be seen as a $\mT$-monoid is recalled in Proposition \ref{prop:T-AlgToT-Mon}, and the monoidal structure induced on  $\Mons{\mT}$ is described in Section \ref{ssec:BoxProd}. The notation $[X,V]$ designates the set of all $\mT$-monoid homomorphisms $f:(X,\alpha)\to(V,q^\ast)$, and the map $\conv:T[X,V]\to[X,V]$ is defined in \ref{ssec:conv}.

We illustrate the various notions of Section \ref{sec:Defs} with the powerset and filter monads, and apply our main result to them at the end of Section \ref{sec:Exp}. In Section \ref{sec:Examples}, we give further examples.

\section{Monads and monoidal structures}\label{sec:Defs}%%%%%%%%%%%%%%%%%%%%%%%%

Classical references for the categorical concepts that we use here include \cite{AdaHerStr:90}, \cite{Bor:94a,Bor:94b} and \cite{Mac:71}. We also include a number of more specific references  in the text.

\subsection{Eilenberg--Moore algebras and sup-semilattices}\label{EMAlgsAndSup}

Let us denote by $\mP=(P,\bigcup,\{-\})$ the powerset monad on $\Set$, and suppose that $\mT$ is a monad with a monad morphism $\tau:\mP\to\mT$. Recall (for example from \cite[Proposition~4.5.9]{Bor:94b}) that $\tau$ induces a functor $\Set^\tau:\Set^\mT\to\Set^\mP$ between the corresponding Eilenberg--Moore categories: this functor commutes with the forgetful functors to $\Set$ and sends a $\mT$-algebra $(X,a)$ to a $\mP$-algebra $(X,a\cdot\tau_X)$. As
\[
\Set^\mP\cong\Sup
\]
is the category of complete sup-semilattices and sup-preserving maps (see for example~\cite[Proposition 4.6.5]{Bor:94b}), any $\mT$-algebra $(X,a)$ has an underlying complete sup-semilattice structure $a\cdot\tau_X$. It follows that a $\mT$-algebra homomorphism $f:(X,a)\to(Y,b)$ has a right adjoint $f^\ast:Y\to X$, so that
\[
1_X\le f^\ast\cdot f\dand f\cdot f^\ast\le 1_Y\ ,
\]
where $X$ and $Y$ are equipped with their induced orders $(a\cdot\tau_X)^\ast:X\to PX$ and $(b\cdot\tau_Y)^\ast:Y\to PY$ respectively, and the hom-sets $\Set(X,Y)$ are ordered pointwise. Note that a $\mT$-algebra structure $a:TX\to X$ is in particular a split epimorphism in $\Set$, so that one has 
\[
1_{TX}\le a^\ast\cdot a\dand a\cdot a^\ast=1_X\ .
\]

\begin{ex}\label{ex:FilterMonad}
The functor $F$ of the \df{filter monad} $\mF=(F,\mu,\eta)$ associates to a set $X$ the set $FX$ of filters on $X$ (that is, subsets of $PX$ that are closed under finite intersection and up-closure), and to a map $f:X\to Y$, the map $Ff:FX\to FY$ defined for all $B\subseteq Y$ and $\fx\in FX$ by
\[
B\in Ff(\fx)\iff f^\inv(B)\in\fx\ .
\]
The components of the unit and multiplication can be described for $x\in X$, $A\subseteq X$, and $\fF\in FFX$ by
\[
A\in\eta_X(x)\iff x\in A\qquad\text{and}\qquad A\in\mu_X(\fF)\iff A^\mF\in\fF\ ,
\]
where $A^\mF:=\{\fx\in FX\mid A\in\fx\}$. The Eilenberg--Moore category of $\mF$ is the category of continuous lattices with continuous sup-preserving maps \cite{Day:75}:
\[
\Set^\mF\cong\Cnt\ .
\]
There is a monad morphism $\tau:\mP\to\mF$, namely the  \df{principal filter monad morphism}, that sends $A\in PX$ to its up-closure $\tau_X(A)=\{B\subseteq X\mid A\subseteq B\}\in FX$. Via $\tau$, any $\mF$-algebra is a complete sup-semilattice, and any $\mF$-algebra homomorphism has a right adjoint. Note that $\tau$ equips the sets $FX$ with the \df{refinement order} $(\mu_X\cdot\tau_{FX})^\ast:FX\to PFX$, that is, the \emph{reverse} inclusion of filters:
\[
\fx\le\fy\iff\fx\supseteq\fy
\]
for all $\fx,\fy\in FX$.
\end{ex}

\subsection{Powerset-enriched monads}\label{ssec:P-EnrichedMnd}%%%%%%%%%%%%%%%%%%%

A \df{powerset-enriched} monad is a pair $(\mT,\tau)$ in which $\mT=(T,\mu,\eta)$ is a monad on $\Set$ and $\tau:\mP\to\mT$ is a monad morphism (from the powerset monad $\mP$) that makes the Kleisli category $\Set_\mT$ into an ordered category; since composition in $\Set_\mT$ is given by
\[
g\kleisli h=\mu_Y\cdot Tg\cdot h
\]
(for $h:Z\to TX$ and $g:X\to TY$) and the hom-sets $\Set(X,TY)$ are ordered pointwise, the condition for $(\mT,\tau)$ to be powerset-enriched is simply
\begin{equation}\label{eq:EnrichmentCond}\tag{$\ast$}
f\le g\implies\mu_Y\cdot Tf\le\mu_Y\cdot Tg
\end{equation}
for all $f,g:X\to TY$. A \df{morphism} $\theta:(\mS,\sigma)\to(\mT,\tau)$ of powerset-enriched monads is a morphism in the comma-category $(\mP\comma\Mnd_\Set)$ of monads under $\mP$; that is, the monad morphism $\theta:\mS\to\mT$ makes the diagram
\[
\xymatrix@C=10pt{&\mP\ar[dl]_\sigma\ar[dr]^\tau&\\ \mS\ar[rr]^{\theta}&&\mT}
\]
commute. When working with powerset-enriched monads $(\mT,\tau)$, we will often assume a fixed choice of $\tau$, and speak of ``the powerset-enriched monad $\mT$''.

\begin{examples}\label{ex:PMnd}
\item The powerset monad $\mP=(P,\bigcup,\{-\})$ is powerset-enriched via the identity monad morphism $1_\mP:\mP\to\mP$. Hence, $(\mP,1_\mP)$ is an initial object in the category of powerset-enriched monads and their morphisms. The order on the sets $PX$ described in \ref{EMAlgsAndSup} is simply subset inclusion suprema are given by arbitrary union.
\item The filter monad $\mF=(F,\mu,\eta)$ is powerset-enriched via the principal filter monad morphism $\tau:\mP\to\mF$ of Example \ref{ex:FilterMonad}.
\end{examples}

The next result shows that condition \eqref{eq:EnrichmentCond} trickles down to the $\mT$-algebra level.

\begin{prop}\label{prop:T-AlgStructActAsMultiplication}
Let $\mT$ be a powerset-enriched monad. For any $\mT$-algebra $(Y,b)$, and maps $f,g:X\to Y$, one has
\[
f\le g\implies b\cdot Tf\le b\cdot Tg\ .
\]
\end{prop}

\begin{proof}
Observe that
\[
f\le g\implies \mu_Y\cdot T(b^\ast\cdot f)\le\mu_Y\cdot T(b^\ast\cdot g)\implies b\cdot\mu_Y\cdot T(b^\ast)\cdot Tf\le b\cdot \mu_Y\cdot T(b^\ast)\cdot Tg\ .
\]
The claimed result follows because the last statement is equivalent to $b\cdot Tf\le b\cdot Tg$ since $b\cdot\mu_Y=b\cdot Tb$ and $b\cdot b^\ast=1_Y$.
\end{proof}

\subsection{$\mT$-monoids}\label{ssec:KlMon}%%%%%%%%%%%%%%%%%%%%%%%%%%%%%%

Let $\mT=(T,\eta,\mu)$ be a monad powerset-enriched via $\tau:\mP\to\mT$. The category $\Mons{\mT}$ of \df{$\mT$-monoids} (or \df{Kleisli monoids}) has as objects pairs $(X,\alpha)$, with $X$ a set and $\alpha:X\to TX$ is a \df{reflexive} and \df{transitive} map:
\[
\eta_X\le\alpha\ ,\quad \alpha\kleisli \alpha\le \alpha\ ;
\]
a homomorphism $f:(X,\alpha)\to(Y,\beta)$ is a map $f:X\to Y$ satisfying
\[
Tf\cdot \alpha\le \beta\cdot f\ .
\]
In presence of the reflexivity condition, transitivity can be expressed as the equality $\alpha\kleisli\alpha=\alpha$, since
\[
\alpha=\alpha\kleisli\eta_X\le\alpha\kleisli\alpha\le\alpha\ .
\]

\begin{prop}\label{prop:T-AlgToT-Mon}
Let $\mT$ be a powerset-enriched monad. The map sending a $\mT$-algebra $(X,a)$ to the pair $(X,a^\ast)$ defines a functor $L:\Set^\mT\to\Mons{\mT}$ that commutes with the underlying-set functors.
\end{prop}

\begin{proof}
Since $a\cdot\eta_X\le 1_X$ and $a\cdot\mu_X\le a\cdot Ta$, one has
\[
\eta_X\le a^\ast\qquad\text{and}\qquad \mu_X\cdot T(a^\ast)\cdot a^\ast\le (a^\ast\cdot a\cdot Ta)\cdot T(a^\ast)\cdot a^\ast=a^\ast\ .
\]
A similar verification shows that a morphism $f:(X,a)\to(Y,b)$ of $\mT$-algebras yields a homomorphism of $\mT$-monoids $f:(X,a^\ast)\to(Y,b^\ast)$.
\end{proof}

By the previous proposition, $(TX,\mu_X^\ast)$ is a $\mT$-monoid, and the idempotency condition makes a $\mT$-monoid structure $\alpha$ on $X$ into a homomorphism $\alpha:(X,\alpha)\to(TX,\mu_X^\ast)$; indeed,
\[
\alpha\kleisli\alpha\le\alpha\iff\mu_X\cdot T\alpha\cdot\alpha\le\alpha\iff T\alpha\cdot\alpha\le\mu_X^\ast\cdot\alpha\ .
\]

\begin{examples}
\item In the case of the powerset monad (together with its identity structure $1_\mP$), $\Mons{\mP}$ is the category of preordered sets. Indeed, a map $\alpha:X\to PX$ is precisely a relation on $X$, and the reflexivity and transitivity conditions make this relation into a preorder; $\alpha_X$ is therefore the down-set map $\dc_X:X\to PX$. A map $f:X\to Y$ is a morphism of  $\Mons{\mP}$ if and only if it preserves the relations, that is, if and only if $f$ is a monotone map. Hence,
\[
\Mons{\mP}\cong\PrOrd\ .
\]
\item When $\mF$ is equipped with the principal filter monad morphism, $\Mons{\mF}$ is the category of topological spaces and continuous maps \cite{Gae:88}:
\[
\Mons{\mF}\cong\Top\ .
\]
Indeed, a Kleisli monoid $\alpha:X\to FX$ associates to each point $x\in X$ a filter $\alpha(x)\in FX$; the elements of this filter all contain $x$ by reflexivity of $\alpha$, and transitivity translates as the axiom required of a family of filters $(\alpha(x))_{x\in X}$ to form the set of neighborhoods of the topology it determines (\cite[Proposition 1.2.2]{Bou:89}). A map $f:X\to Y$ is a morphism of  $\Mons{\mF}$ if and only if the image of the neighborhood $\alpha(x)$ of $x\in X$ is finer that the neighborhood $\beta(f(x))$ of $f(x)$, that is, if and only if $f$ continuous.
\end{examples}

\begin{prop}\label{prop:T-MonTopological}
For a powerset-enriched monad $\mT$, the forgetful functor $\Mons{\mT}\to\Set$ is topological: given a family $(Y_i,\beta_i)$  of $\mT$-monoids, the initial structure induced on $X$ by a family of maps $(f_i:X\to Y_i)_{i\in I}$ is
\[
\alpha:=\bigwedge_{i\in I}(Tf_i)^\ast\cdot\beta_i\cdot f_i\ .
\]
\end{prop}

\begin{proof}
The claim follows from direct verifications (see \cite[Remark 3.11]{Sea:10}).
\end{proof}

\begin{examples}
\item When $\mT$ is the powerset monad $\mP$, Proposition \ref{prop:T-AlgToT-Mon} simply described the functor
\[
\Sup\cong\Set^\mP\to\Mons{\mP}\cong\PrOrd\ .
\]
that send a complete sup-semilattice to its underlying preordered set. Proposition \ref{prop:T-MonTopological} recalls that the underlying-set functor $\PrOrd\to\Set$ is topological \cite{AdaHerStr:90}.
\item In the case where $\mT$ is the filter monad $\mF$, the functor
\[
\Cnt\cong\Set^\mF\to\Mons{\mF}\cong\Top
\]
of Proposition \ref{prop:T-AlgToT-Mon} sends a continuous lattice to its underlying topological space by equipping it with the Scott topology. In turn, Proposition \ref{prop:T-MonTopological} yields the well-known fact that the forgetful functor $\Top\to\Set$ is topological.
\end{examples}

\subsection{Monoidal functors}%%%%%%%%%%%%%%%%%%%%%%%%%%%%%%%

Let us denote by
\[
\upsilon_{X,Y,Z}:X\times(Y\times Z)\to(X\times Y)\times Z\ ,\quad\lambda_X:X\to X\times 1\quad\rho_X:X\to 1\times X%\\
%\dand\zeta_{X,Y}:X\times Y\to Y\times X
\]
the associativity and unitality natural isomorphisms that form  the cartesian structure of $\Set$. Consider a     functor $T:\Set\to\Set$, a map $\eta_1:1\to T1$ (with $1=\{\star\}$ a singleton), and a family of maps
\[
\kappa=(\kappa_{X,Y}:TX\times TY\to T(X\times Y))_{X,Y\in\ob\Set}
\]
natural in $X$ and $Y$. The functor $T$ is \df{monoidal} (with respect to $(\eta_1,\kappa)$) if it is compatible with the cartesian structure of $\Set$ as follows:

\begin{enumerate}[label=$(M_{\arabic{*}})$]%,ref=$M_{\arabic{*}}$]
\item\label{cond:0} one has $\kappa_{X\times Y,Z}\cdot(\kappa_{X,Y}\times 1_{TZ})\cdot\upsilon_{TX,TY,TZ}=T\upsilon_{X,Y,Z}\cdot\kappa_{X,Y\times Z}\cdot(1_{TX}\times\kappa_{Y,Z})$:
\[
\xymatrix{TX\times(TY\times TZ)\ar[d]_{\upsilon}\ar[r]^-{1\times\kappa}&TX\times T(Y\times Z)\ar[r]^-{\kappa}&T(X\times(Y\times Z))\ar[d]^{T\upsilon}\\
(TX\times TY)\times TZ\ar[r]^-{\kappa\times 1}&T(X\times Y)\times TZ\ar[r]^-{\kappa}&T((X\times Y)\times Z)\rlap{\ ;}}
\]
\item\label{cond:1} one has $\kappa_{X,1}\cdot(1_{TX}\times\eta_1)\cdot\lambda_{TX}=T\lambda_X$ and $\kappa_{1,X}\cdot(\eta_1\times1_{TX})\cdot\rho_{TX}=T\rho_X$:
\[
\xymatrix@C=2.5em{&TX\ar[dl]_{\lambda}\ar[dr]^{T\lambda}&\\
TX\times 1\ar[r]^-{1\times\eta_1}&TX\times T1\ar[r]^-{\kappa}&T(X\times 1)}
\xymatrix{\mbox{}\ar@{}[d]|{\textstyle{\text{and}}}\\
\mbox{}}
\xymatrix@C=2.5em{&TX\ar[dl]_{\rho}\ar[dr]^{T\rho}&\\
1\times TX\ar[r]^-{\eta_1\times1}&T1\times TX\ar[r]^-{\kappa}&T(1\times X)\rlap{\ .}}
\]
\end{enumerate}

\begin{examples}\label{ex:MonoidalTransformation}
\item\label{ex:MonoidalTransformation:P} There is a natural transformation $\pi=(\pi_{X,Y}:PX\times PY\to P(X\times Y))_{X,Y\in\Set}$ associated to the powerset functor $P:\Set\to\Set$  given by
\[
\pi_{X,Y}(A,B)=A\times B
\]
(for all $A\in PX$ and $B\in PY$). The pair $(P,\{-\}_1)$ is a monoidal functor (where $\{-\}_1:1\to P1$ is the singleton map).
\item The product of the filters $\fx\in FX$ and $\fy\in FY$ is defined by
\[
\fx\times\fy:=\{C\subseteq X\times Y\mid\exists A\in\fx,B\in\fy\ (A\times B\subseteq C)\}\in F(X\times Y)\ .
\]
One obtains the \df{product filter} natural transformation $\kappa:F(-)\times F(-)\to F(-\times -)$ by setting
\[
\kappa_{X,Y}(\fx,\fy)=\fx\times \fy
\]
(for all $\fx\in FX$ and $\fy\in FY$). The pair $(F,\eta_1)$ then forms a monoidal functor (where $\eta_1:1\to F1$ is the principal ultrafilter map).
\end{examples}

\subsection{Lax monoidal monads}\label{ssec:MonoidalP-Mnd}

A powerset-enriched monad $\mT$ with a natural transformation $\kappa:T(-)\times T(-)\to T(-\times -)$ is \df{lax monoidal} if $T$ is monoidal with respect to $(\eta_1,\kappa)$ and $\kappa$ is laxly compatible with the monad structures as follows:

\begin{enumerate}[label=$(M_{\arabic{*}})$]\setcounter{enumi}{2}
\item\label{cond:unit} one has $\eta_{X,Y}\le\kappa_{X,Y}\cdot(\eta_X\times\eta_Y)$:
\[
\xymatrix{&X\times Y\ar[dl]_{\eta\times\eta}\ar[dr]^{\eta}\ar@{}[d]|-{\ge}&\\
TX\times TY\ar[rr]^-{\kappa}&&T(X\times Y)\rlap{\ ;}}
\]
\item\label{cond:3} one has $\mu_{X\times Y}\cdot T\kappa_{X,Y}\cdot\kappa_{TX,TY}\le\kappa_{X,Y}\cdot(\mu_X\times\mu_Y)$:
\[ 
\xymatrix@C=4em{TTX\times TTY\ar@{}[dr]|-{\ge}\ar[r]^-{T\kappa\cdot\kappa}\ar[d]_-{\mu\times\mu}&TT(X\times Y)\ar[d]^-{\mu}\\
TX\times TY\ar[r]^{\kappa}&T(X\times Y)\rlap{\ ;}}
\]
\item\label{cond:enr} one has $\tau_X\cdot\pi_{X,Y}\le\kappa_{X,Y}\cdot(\tau_X\times\tau_Y)$:
\[ 
\xymatrix@C=4em{PX\times PY\ar@{}[dr]|-{\ge}\ar[r]^-{\pi}\ar[d]_-{\tau\times\tau}&P(X\times Y)\ar[d]^-{\tau}\\
TX\times TY\ar[r]^{\kappa}&T(X\times Y)\rlap{\ ,}}
\]
were $\pi:P(-)\times P(-)\to P(-\times -)$ is the natural transformation associated to the powerset monad (Example \ref{ex:MonoidalTransformation}\ref{ex:MonoidalTransformation:P}).
\end{enumerate}
When referring to a lax monoidal powerset-enriched monad $\mT$, we will always assume that the natural transformation $\kappa$ is given together with $\tau$. Hence, a lax monoidal powerset-enriched monad $\mT$ is in effect a triple $(\mT,\tau,\kappa)$ with $\tau:P\to T$ and $\kappa:T(-)\times T(-)\to T(-\times -)$ suitable natural transformations.

Note that conditions \ref{cond:3} with \ref{cond:enr} yield
\begin{align*}
\mu_X\cdot \tau_{TX}\cdot P\kappa_X\cdot\pi_{TX,TY}&=\mu_X\cdot T\kappa_X\cdot\tau_{TX}\cdot\pi_{TX,TY}\\
&\le\mu_X\cdot T\kappa_X\cdot\kappa_{TX,TY}\cdot(\tau_{TX}\times\tau_{TY})\\
&\le\kappa_{X,Y}\cdot(\mu_X\times\mu_Y)\cdot(\tau_{TX}\times\tau_{TY})=\kappa_{X,Y}\cdot((\mu_X\cdot\tau_{TX})\times(\mu_Y\cdot\tau_{TY}))\ .
\end{align*}
By adjunction, one obtains the diagram
\[ 
\xymatrix@C=4em{PTX\times PTY\ar@{}[dr]|-{\le}\ar[r]^-{P\kappa\cdot\pi}&PT(X\times Y)\\
TX\times TY\ar[u]^{(\mu\cdot\tau)^\ast\times(\mu\cdot\tau)^\ast}\ar[r]^{\kappa}&T(X\times Y)\ar[u]_{(\mu\cdot\tau)^\ast}\ ;}
\]
Since $(\mu_X\cdot\tau_X)^\ast:TX\to PTX$ is the down-set map associated to the order induced by $\tau$ on $TX$, this diagram means that $\kappa$ preserves the order induced by $\tau$ in each variable:
\[
(\fx,\fy)\le(\fx',\fy')\implies\kappa_{X,Y}(\fx,\fy)\le\kappa_{X,Y}(\fx',\fy')
\]
for all $\fx,\fx'\in TX$, $\fy,\fy'\in TY$.

\begin{examples}\label{ex:LaxMonoidalMonad}
\item\label{ex:LaxMonoidalMonad:P} The powerset monad $\mP$ together with its powerset-enrichment $1_\mP:\mP\to\mP$ and the natural transformation $\pi:P(-)\times P(-)\to P(-\times -)$ of Example \ref{ex:MonoidalTransformation}\ref{ex:MonoidalTransformation:P} is a lax monoidal powerset-enriched monad.
\item The filter monad $\mF$ with the principal filter monad morphism $\tau:\mP\to\mF$ and the product filter natural transformation $\kappa:F(-)\times F(-)\to F(-\times -)$ is a lax monoidal powerset-enriched monad.
\end{examples}

%\begin{rem}
%Instead of requiring that $\kappa_{X,Y}$ be monotone in each variable, it would be natural to ask that the diagram
%\[ 
%\xymatrix@C=4em{PTX\times PTY\ar[r]^-{P\kappa\cdot\pi}\ar[d]_{(\mu\cdot\tau)\times(\mu\cdot\tau)}&PT(X\times Y)\ar[d]^{\mu\cdot\tau}\\
%TX\times TY\ar[r]^{\kappa}&T(X\times Y)}
%\]
%commutes, that is, $\kappa_{X,Y}$ preserves suprema in each variable. However, this is not reasonable in our setting, as the filter monad does not satisfy this condition.
%
%Indeed, consider the sets $X=Y=\mathbf{N}$ (where $\mathbf{N}$ denotes the set of natural numbers). Let $\fx\in FX$ be the filter on $\mathbf{N}$ with basis $A_i:=\{n\in X\mid i\le n\}$ ($i\in\mathbf{N}$), and $\fy_j\in FY$ the principal ultrafilter on $j\in\mathbf{N}$. On one hand, the set $S:=\bigcup_{j\in\mathbb{N}}A_{j}\times\{j\}\subseteq X\times Y$ is an element of $\bigvee_{j\in\mathbf{N}}(\fx\times\fy_j)$. On the other hand, $\bigvee_{j\in\mathbf{N}}\fy_j=Y$, so the filter $\fx\times\bigvee_{j\in\mathbf{N}}\fy_j$ is spanned by elements of the form $A_i\times Y$ with $A_i\in\fx$; hence, none of these is a subset of $S$, and $S\notin\fx\times\bigvee_{j\in\mathbf{N}}\fy_j$.
%\end{rem}

\subsection{The box products of $\mT$-monoids}\label{ssec:BoxProd}%%%%%%%%%%%%%%%%
Let $\mT$ be a lax monoidal powerset-enriched monad. The \df{box product} $(X,\alpha)\xtimes(Y,\beta)$ of the $\mT$-monoids $(X,\alpha)$ and $(Y,\beta)$ is obtained by equipping the set $X\times Y$ with the structure $\alpha\xtimes\beta:X\times Y\to T(X\times Y)$ defined by
\[
\alpha\xtimes\beta:=\kappa_{X,Y}\cdot(\alpha\times\beta)\ .
\]

\begin{examples}
\item For the powerset-enriched monoidal powerset monad $\mP$, the box product of preordered sets $(X,\dc_X)$ and $(Y,\dc_Y)$ is the relation on $X\times Y$ that sends a pair $(x,y)\in X\times Y$ to the set
\[
\pi_{X,Y}\cdot(\dc_X\times\dc_Y)(x,y)=\dc_X x\times\dc_Y y\ .
\]
In other words,
\[
(x',y')\le (x,y)\iff x\le x'\ \&\ y\le y'
\]
(for all $x,x'\in X$, $y,y'\in Y$), that is, the box product of preordered sets is their ordinary product.
\item For the powerset-enriched monoidal filter monad $\mF$, the box product of topological spaces $(X,\alpha)$ and $(Y,\beta)$ is given by the topology on $X\times Y$ that associates to a point $(x,y)\in X\times Y$ their product neighborhood filter
\[
\kappa_{X,Y}\cdot(\alpha\times\beta)(x,y)=\alpha(x)\times\alpha(y)\ .
\]
Hence, the box product of topological spaces is their ordinary product. In Section \ref{sec:Examples}, we will present categories of $\mT$-monoids for which the box product is not the categorical product.
\end{examples}

\begin{lem}
The box product of $\mT$-monoids defines a functor
\[
(-)\xtimes(-):\Mons{\mT}\times\Mons{\mT}\to\Mons{\mT}
\]
that commutes with the underlying-set functors.
\end{lem}

\begin{proof}
For $\mT$-monoid morphisms $f:(X,\alpha)\to(X',\alpha,)$, $g:(Y,\beta)\to(Y',\beta')$, we set $f\xtimes g=f\times g$. Since
\[
T(f\times g)\cdot\kappa_{X,Y}\cdot(\alpha\times\beta)=\kappa_{X',Y'}\cdot(Tf\times Tg)\cdot(\alpha\times\beta)\le \kappa_{X',Y'}\cdot(\alpha'\times\beta')\cdot(f\times g)
\]
by naturality and monotonicity of $\kappa$, $(-)\xtimes(-)$ is well-defined, and is therefore a functor.
\end{proof}

\begin{prop}
Let $\mT$ be a lax monoidal powerset-enriched monad. With $E:=(1,\eta_1)$, the triple $(\Mons{\mT},\xtimes,E)$ is a monoidal category whose underlying structure maps are those of the cartesian structure of $\Set$.
\end{prop}

\begin{proof}
It suffices to prove that the components of the natural transformations $\upsilon$, $\lambda$ and $\rho$ are isomorphisms of $\mT$-monoids. Thus, let $(X,\alpha)$, $(Y,\beta)$ and $(Z,\gamma)$ be $\mT$-monoids. With the use of \ref{cond:0}, the equalities
\begin{align*}
&T\upsilon_{X,Y,Z}\cdot\kappa_{X,Y\times Z}\cdot(\alpha\times(\kappa_{Y,Z}\cdot(\beta\times\gamma)))\\
&=T\upsilon_{X,Y,Z}\cdot\kappa_{X,Y\times Z}\cdot(1_{TX}\times\kappa_{Y,Z})\cdot(\alpha\times(\beta\times\gamma))\\
&=\kappa_{X\times Y,Z}\cdot(\kappa_{X,Y}\times 1_{TZ})\cdot\upsilon_{TX,TY,TZ}\cdot(\alpha\times(\beta\times\gamma))\\
&=\kappa_{X\times Y,Z}\cdot(\kappa_{X,Y}\times 1_{TZ})\cdot((\alpha\times\beta)\times\gamma)\cdot\upsilon_{X,Y,Z}\\
&=\kappa_{X\times Y,Z}\cdot((\kappa_{X,Y}\cdot(\alpha\times\beta))\times\gamma)\cdot\upsilon_{X,Y,Z}
\end{align*}
show that the associativity map $\upsilon_{X,Y,Z}$ is a $\mT$-monoid isomorphism. Similarly, \ref{cond:1} yields that $\lambda_X:X\to 1\times X$ and $\rho_X:X\to X\times 1$ are isomorphisms.
\end{proof}

\section{Exponentiability}\label{sec:Exp}%%%%%%%%%%%%%%%%%%%%%%%%%%%%%%%%%%

\subsection{Exponentiable $\mT$-monoids}%%%%%%%%%%%%%%%%%%%%%%%%%%%%%%%%
Given a lax monoidal powerset-enriched monad $\mT$, we say that a $\mT$-monoid $(X,\alpha)$ is \df{exponentiable} if the functor $(-)\xtimes X:\Mons{\mT}\to\Mons{\mT}$ has a right adjoint $G_X$. In Lemma \ref{lem:UnderlyingExponential}, we prove that the underlying set of an \df{exponential} $G_X Y$ can be identified with the set
\[
[X,Y]:=\Mons{\mT}(X,Y)
\]
of all $\mT$-monoid homomorphisms $f:X\to Y$ (with $(Y,\beta)$ a $\mT$-monoid). Hence, we will often write $[X,-]$ in lieu of $G_X$ (with the $\mT$-monoid structure $\nbhd:[X,Y]\to T[X,Y]$ tacitly assumed when necessary). To further simplify notations, we write $Y^X$ for the hom-set $\Set(X,Y)$, so the counit of the adjunction
\[
((-)\times X)\dashv(-)^X:\Set\to\Set
\]
is given by the evaluation maps
\[
\ev_{X,Y}:Y^X\times X\to Y
\]
(for all $X,Y\in\ob\Set$). We also denote by $\ev_{X,Y}$ the restriction $\ev_{X,Y}:[X,Y]\times X\to Y$.

Note that we explicitly study the right adjoint of $(-)\xtimes X$ only: the right adjoint of $X\xtimes(-)$ can be obtained \textit{mutatis mutandis}, once the case for $(-)\xtimes X$ has been elucidated. 

\subsection{Exponential $\mT$-monoids are $\mT$-algebras}%%%%%%%%%%%%%%%%%%%%%%
%The aim of this section is to prove that the exponentiable objects $(X,\alpha)$ in $(\Mons{\mT},E,\xtimes)$ have exponentials of the form $([X,V],\nbhd)$ whenever $(V,q)$ is a $\mT$-algebra.

\begin{lem}\label{lem:UnderlyingExponential}
Let $\mT$ be a lax monoidal powerset-enriched monad, and $E=(1,\eta_1)$. There is a natural bijection between $[E,-]$ and the underlying-set functor $|-|:\Mons{\mT}\to\Set$.

In particular, for an exponentiable $\mT$-monoid $(X,\alpha)$, the corresponding exponentials can be considered to be sets of the form $[X,Y]$ with a $\mT$-monoid structure $\nbhd:[X,Y]\to T[X,Y]$.
\end{lem}

\begin{proof}
Any map $x:1\to X$ is a homomorphism $x:(1,\eta_1)\to(X,\alpha)$:
\[
Tx\cdot\eta_1=\eta_X\cdot x\le\alpha\cdot x\ ,
\]
so that elements of $X$ are in bijective correspondence with elements of $[E,X]$: $[E,X]\cong X^1\cong X$. It is also clear that the bijection $[E,X]\to X$ is natural in $X$.

Suppose that $(X,\alpha)$ is exponentiable in $\Mons{\mT}$, and denote by $G_X$ the right adjoint to $(-)\xtimes X:\Mons{\mT}\to\Mons{\mT}$. By the previous point, there are bijections
\[
[X,Y]\cong[E\xtimes X,Y]\cong[E,G_X Y]\cong|G_X Y|
\]
that are natural in $Y$, so that one can always replace $|G_X Y|$ by $[X,Y]$. 
\end{proof}

\begin{lem}\label{lem:ExOfConv}
Consider a lax monoidal powerset-enriched monad $\mT$, and a $\mT$-algebra $(V,q)$. Let $(X,\alpha)$ be  an exponentiable $\mT$-monoid, and $([X,V],\nbhd)$ denote the $G_X$-image of the $\mT$-monoid $(V,q^\ast)$. Then there is a unique $\mT$-monoid homomorphism $\conv:(T[X,V],\mu_{[X,V]}^\ast)\to([X,V],\nbhd)$ that makes the following diagram commute:
\[
\xymatrix@C=30pt{[X,V]\xtimes X\ar[r]^-{\ev}&V\\
T[X,V]\xtimes X\ar[u]^{\conv\xtimes 1_X}\ar[ur]<-1pt>_(0.6){\quad q \cdot T\ev\cdot\kappa\cdot(1\times\alpha)}&}
\]
\end{lem}

\begin{proof}
The fact that the underlying set of the exponential object $G_XV$ is $[X,V]$ follows from Lemma \ref{lem:UnderlyingExponential}. The map $\kappa_{[X,V],X}\cdot(1_{T[X,V]}\times\alpha):T[X,V]\xtimes X\to T([X,V]\times X)$ is a $\mT$-monoid homomorphism:
\begin{align*}
T(\kappa_{[X,V],X}&\cdot(1_{T[X,V]}\times\alpha))\cdot\kappa_{T[X,V],X}\cdot(\mu_{[X,V]}^\ast\times\alpha)\\
&= T\kappa_{[X,V],X}\cdot\kappa_{[X,V],X}\cdot(1_{TT[X,V]}\times T\alpha)\cdot(\mu_{[X,V]}^\ast\times\alpha)\\
&\le T\kappa_{[X,V],X}\cdot\kappa_{[X,V],X}\cdot(\mu_{[X,V]}^\ast\times\mu_X^\ast)\cdot(1_{T[X,V]}\times\alpha)\\
&\le\mu_{[X,V]\times X}^\ast\cdot\kappa_{[X,V],X}\cdot(1_{T[X,V]}\times\alpha)\ ,
\end{align*}
and consequently so is the composite $g:=q \cdot T\ev_{X,V}\cdot\kappa_{[X,V],X}\cdot(1_{T[X,V]}\times\alpha):T[X,V]\xtimes X\to V$. The couniversal property of $\ev_{X,V}:[X,V]\times X\to V$ then yields a unique $\mT$-monoid homomorphism $\conv:T[X,V]\to[X,V]$ with $\ev_{X,V}\cdot(\conv\times 1_X)=g$, as claimed.
\end{proof}

\begin{thm}\label{thm:1}
Let $\mT$ be a lax monoidal powerset-enriched monad, and $(V,q)$ a $\mT$-algebra. If $(X,\alpha)$ is an exponentiable $\mT$-monoid, then $([X,V],\conv)$ is a $\mT$-algebra.
\end{thm}

\begin{proof}
The exponentiable $\mT$-monoid $(X,\alpha)$ yields a $\mT$-monoid structure $\nbhd:[X,V]\to T[X,V]$ that makes $\ev_{X,V}:[X,V]\times X\to V$ a $\mT$-monoid homomorphism:
\[
T\ev_{X,V}\cdot\kappa_{[X,V],X}\cdot(\nbhd\times\alpha)\le q^\ast\cdot\ev_{X,V}\ .
\]
All the maps composing $g:=q \cdot T\ev_{X,V}\cdot\kappa_{[X,V],X}\cdot(1_{T[X,V]}\times\alpha)$ are monotone and $[X,V]$ is ordered pointwise, so the defining equality  $\ev_{X,V}\cdot(\conv\times 1_X)=g$ (Lemma \ref{lem:ExOfConv}) implies that $\conv$ is monotone. Moreover,
\begin{align*}
\ev_{X,V}&=q \cdot\eta_V\cdot\ev_{X,V}=q \cdot T\ev_{X,V}\cdot\eta_{[X,V]\times X}\\
&\le q \cdot T\ev_{X,V}\cdot\kappa_{[X,V],X}\cdot(\eta_{[X,V]}\times\eta_X)\\
&\le q \cdot T\ev_{X,V}\cdot\kappa_{[X,V],X}\cdot(\eta_{[X,V]}\times\alpha)=\ev_{X,V}\cdot(\conv\times 1_X)\cdot(\eta_{[X,V]}\times 1_X)\ ,
\end{align*}
that is, $1_{[X,V]}\le\conv\cdot\eta_{[X,V]}$. The unicity of $\mT$-monoid homomorphisms $[X,V]\to[X,V]$ in the couniversal property of $\ev_{X,V}:[X,V]\times X\to V$ also forces $\conv\cdot\nbhd=1_{[X,V]}$ because $g\cdot(\nbhd\times 1_X)=\ev_{X,V}$:
\[
\ev_{X,V}\le q\cdot T\ev_{X,V}\cdot\kappa_{[X,V],X}\cdot(\eta_{[X,V]}\times\eta_X)\le g\cdot(\nbhd\times 1_X)\le q\cdot q^\ast\cdot\ev_{X,V}=\ev_{X,V}\ .
\]
Hence, monotonicity of $\conv$ and $\eta_{[X,V]}\le\nbhd$ yield
\[
\conv\cdot\eta_{[X,V]}=\conv\cdot\nbhd=1_{[X,V]}\ .
\]
Since $\conv$ is a $\mT$-monoid homomorphism and $\mu_{[X,V]}\cdot T\eta_{[X,V]}=1_{T[X,V]}$, one has $1_{T[X,V]}=T\conv\cdot T\eta_{[X,V]}\le T\conv\cdot\mu_{[X,V]}^\ast\le\nbhd\cdot\conv$, so
\[
1_{T[X,V]}\le\nbhd\cdot\conv\dand \mu_{[X,V]}\le\mu_{[X,V]}\cdot T\nbhd\cdot T\conv\ .
\]
By transitivity of $\nbhd$, one obtains  $\conv\cdot\mu_{[X,V]}\cdot T\nbhd\le\conv\cdot\mu_{[X,V]}\cdot T\nbhd\cdot\nbhd\cdot\conv=\conv\cdot\nbhd\cdot\conv=\conv$, so
\[
\conv\cdot\mu_{[X,V]}\le\conv\cdot\mu_{[X,V]}\cdot T\nbhd\cdot T\conv\le\conv\cdot T\conv\ .
\]
With
\[
\conv\cdot T\conv\le\conv\cdot T\conv\cdot\mu_{[X,V]}^\ast\cdot\mu_{[X,V]}\le\conv\cdot\nbhd\cdot\conv\cdot\mu_{[X,V]}\ ,
\]
one concludes that
\[
\conv\cdot\mu_{[X,V]}=\conv\cdot T\conv
\]
and $([X,V],\conv)$ is a $\mT$-algebra.
\end{proof}

\subsection{$\mT$-algebraic hom-sets are exponential}%%%%%%%%%%%%%%%%%%%%%%%%%%
Let $\mT$ be a lax monoidal powerset-enriched monad, and $\sV$ a family of $\mT$-monoids. We say that a $\mT$-monoid $(X,\alpha)$ is \df{$\sV$-exponentiable} if for all $(V,q)\in\sV$, there is a $\mT$-monoid structure $\nbhd_{[X,V]}:[X,V]\to T[X,V]$ that makes $\ev_{X,V}:[X,V]\times X\to V$ into a $((-)\xtimes X)$-couniversal arrow for $V$: for all $g:Y\xtimes X\to V$ in $\Mons{\mT}$, there is exactly one $\mT$-monoid homomorphism $f:Y\to[X,V]$ with $\ev_{X,V}\cdot(f\times 1_X)=g$
\[
\xymatrix{[X,V]\\Y\ar@{.>}[u]^f}\qquad
\xymatrix@C=30pt{[X,V]\xtimes X\ar[r]^-{\ev}&V\\
Y\xtimes X\ar[u]^{f\xtimes 1_X}\ar[ur]<-1pt>_g&}
\]
Given a $\sV$-exponentiable $\mT$-monoid $(X,\alpha)$, we define for any $\mT$-monoid $(Y,\beta)$ a $\mT$-monoid structure $\omega:[X,Y]\to T[X,Y]$ via the initial lift of the family $([X,f]:[X,Y]\to[X,V])_{V\in\sV,f\in[Y,V]}$:
 \[
\omega:=\bigwedge\nolimits_{V\in\sV,f\in[Y,V]} (T[X,f])^\ast\cdot\nbhd_{[X,V]}\cdot[X,f]
\]
(see Proposition \ref{prop:T-MonTopological}).

\begin{lem}\label{lem:RestrictionToV}
Let $\mT$ be a lax monoidal powerset-enriched monad, $\sV$ a family of $\mT$-monoids, $(X,\alpha)$ a $\sV$-exponentiable $\mT$-monoid, and $(Y,\beta)$ a $\mT$-monoid. If $\sV$ is initially dense in $\Mons{\mT}$, then $\ev_{X,Y}:[X,Y]\xtimes X\to Y$ is a $\mT$-monoid homomorphism.
\end{lem}

\begin{proof}
By applying $T$ to the equality $f\cdot\ev_{X,Y}=\ev_{X,V}\cdot([X,f]\times 1_X)$ for all $f\in[Y,V]$, we obtain that the maps $f\cdot\ev_{X,Y}:[X,Y]\xtimes X\to V$ are $\mT$-monoid homorphisms:
\begin{align*}
T(f\cdot\ev_{X,Y})&\cdot\kappa_{X,Y}\cdot(\omega\times\alpha)\\
&\le T(f\cdot\ev_{X,Y})\cdot\kappa_{X,Y}\cdot(((T[X,f])^\ast\cdot\nbhd_{[X,V]}\cdot[X,f])\times\alpha)\\
&\le Tf\cdot T\ev_{X,Y}\cdot(T([X,f]\times 1_{X}))^\ast\cdot\kappa_{X,V}\cdot((\nbhd_{[X,V]}\cdot[X,f])\times\alpha)\\
&\le T\ev_{X,V}\cdot\kappa_{X,V}\cdot(\nbhd_{[X,V]}\times\alpha)\cdot([X,f]\times 1_X)\\
&\le q^\ast\cdot\ev_{X,V}\cdot([X,f]\times 1_X)\\
&= q^\ast\cdot f\cdot\ev_{X,Y}\ .
\end{align*}
As the family $\sV$ is initially dense, $\ev_{X,Y}$ is a $\mT$-monoid homomorphism.
\end{proof}

\begin{lem}\label{lem:ProductToOneVariable}
Let $\mT$ be a lax monoidal powerset-enriched monad. Given $\mT$-monoids $(X,\alpha)$, $(Y,\beta)$, $(Z,\gamma)$ and a homomorphism $g:Z\xtimes X\to Y$, the map $g_z:X\to Y$ is a $\mT$-monoid homomorphism for all $z\in Z$ (where $g_z(x):=g(z,x)$ for all $x\in X$).
\end{lem}

\begin{proof}
If $g:Z\xtimes X\to Y$ is a $\mT$-monoid homomorphism, then, denoting by $z:1\to Z$ the constant map onto $z\in Z$, one has
\begin{align*}
Tg_z\cdot\alpha&=Tg\cdot T(z\times 1_X)\cdot T\rho_X\cdot\alpha\\
&=Tg\cdot T(z\times 1_X)\cdot\kappa_{1,X}\cdot(\eta_1\times 1_{TX})\cdot\rho_{TX}\cdot\alpha\\
&= Tg\cdot\kappa_{Z,X}\cdot(\eta_Z\times1_{TX})\cdot(z\times 1_{TX})\cdot\rho_{TX}\cdot\alpha\\
&\le Tg\cdot\kappa_{Z,X}\cdot(\gamma\times1_{TX})\cdot(1_{Z}\times\alpha)\cdot(z\times 1_X)\cdot\rho_{X}\\
&\le\beta\cdot g\cdot(z\times 1_X)\cdot\rho_{X}=\beta\cdot g_z\ ,
\end{align*}
so that $g_z$ is a $\mT$-monoid homomorphism.
\end{proof}

\begin{prop}\label{prop:RestrictionToV}
Let $\mT$ be a lax monoidal powerset-enriched monad and $\sV$ an initially dense family of $\mT$-monoids. A $\mT$-monoid $(X,\alpha)$ is $\sV$-exponentiable if and only if it is exponentiable.
\end{prop}

\begin{proof}
Obviously, an exponentiable $\mT$-monoid is also $\sV$-exponentiable. To prove the other implication, we need to show that $\ev_{X,Y}:([X,Y]\times X,\kappa_{X,Y}\cdot(\omega\times\alpha))\to(Y,\beta)$ is a $((-)\xtimes X)$-couniversal arrow for any $\mT$-monoid $(Y,\beta)$. Lemma \ref{lem:RestrictionToV} already asserts that $\ev_{X,Y}$ is a homomorphism.

Consider a homomorphism $g:Z\xtimes X\to Y$ and $(V,q)\in\sV$. By Lemma \ref{lem:ProductToOneVariable}, the unique map $f:Z\to Y^X$ with $\ev_{X,Y}\cdot(f\times 1_X)=g$ corestricts to $f:Z\to[X,Y]$. The couniversal property of $\ev_{X,V}$, then produces for any $h\in[Y,V]$ a unique homomorphism $k:Z\to[X,V]$ with $\ev_{X,V}\cdot(k\times 1_X)=h\cdot g$; since
\[
h\cdot g=h\cdot\ev_{X,Y}\cdot(f\times 1_X)=\ev_{X,V}\cdot(([X,h]\cdot f)\times 1_X)\ ,
\]
one necessarily has $[X,h]\cdot f=k:Z\to[X,V]$, that is, $[X,h]\cdot f$ is a homomorphism for all $h\in[Y,V]$. As $\omega$ is the initial $\mT$-monoid structure on $[X,Y]$ with respect to all $h\in[Y,V]$ with $V\in\sV$, one concludes that the unique map $f:Z\to[X,Y]$ is in fact a homomorphism.
\end{proof}

\subsubsection{The convergence map}\label{ssec:conv}%%%%%%%%%%%%%%%%%%%%%%%%%
Let $\mT$ be a lax monoidal powerset-enriched monad and $(V,q)$ a $\mT$-algebra. For a $\mT$-monoid $(X,\alpha)$, we equip the set $[X,V]$ of all homomorphisms $f:(X,\alpha)\to(V,q^\ast)$ with its pointwise order. There is then a map $\conv_{[X,V]}=\conv:T[X,V]\to[X,V]$ defined as follows: for $\ff\in T[X,Y]$,
\begin{align*}
\conv(\,\ff\:\!)&:=\bigwedge\{g\in[X,V]\mid\forall x\in X\ (T\ev_{X,V}\cdot\kappa_{[X,V],X}(\,\ff,\alpha(x))\le q^\ast\cdot\ev_{X,V}(g,x))\}\\
&=\bigwedge\{g\in[X,V]\mid\forall x\in X\ (q\cdot T\ev_{X,V}\cdot\kappa_{[X,V],X}\cdot(1_{[X,V]}\times\alpha)(\,\ff,x)\le\ev_{X,V}(g,x))\}\ ,
\end{align*}
with infimum taken in $V^X$ (Lemma \ref{lem:ConvSmallest} below shows that the codomain of $\conv$ is indeed $[X,V]$). With the pointwise order on $[X,V]^{T[X,V]}$, $\conv$ is then the smallest map $c:T[X,V]\to[X,V]$ that makes the following diagram commute laxly:
\[
\xymatrix@C=3em{\rule{1.5ex}{0ex}[X,V]\times X\ar[r]^-{\ev}_(0.3){}="u"&V\ar@{}[l]^(0.05){\rule{0ex}{0.5ex}}="ur"\\
T[X,V]\times X\ar[u]<-0.5ex>^{c\times 1_X}\ar@{}[r]^(0.25){\rule{0ex}{1ex}}="d"&\ar@{}"d";"u"|(0.5){\ge}\ar"d";"ur"_(0.6){\quad q\cdot T\ev\cdot\kappa\cdot(1\times\alpha)}}
\]

\begin{lem}\label{lem:ConvSmallest}
Let $\mT$ be a powerset-enriched monad, $(X,\alpha)$ a $\mT$-monoid, and $(V,q)$ a $\mT$-algebra. With $[X,V]$ ordered pointwise, one has ${\bigwedge S}\in [X,V]$ for all $S\subseteq [X,V]$, that is, the order-embedding $[X,V]\into V^X$ creates infima.
\end{lem}

\begin{proof}
Let $S\subseteq[X,V]$. For any $g\in S$, we have $Tg\cdot\alpha\le q^\ast\cdot g$, or equivalently $q\cdot Tg\cdot\alpha\le g$. By Proposition \ref{prop:T-AlgStructActAsMultiplication},
\[
q\cdot T(\textstyle{\bigwedge S})\cdot\alpha\le\bigwedge_{g\in S}q\cdot Tg\cdot\alpha\le\bigwedge S\ ,
\]
so $\bigwedge S$ is a $\mT$-monoid homomorphism: $T(\bigwedge S)\cdot\alpha\le q^\ast\cdot\bigwedge S$.
\end{proof}

\begin{prop}\label{prop:thm:2}
Let $\mT$ be a lax monoidal powerset-enriched monad, $(V,q)$ a $\mT$-algebra and $(X,\alpha)$ a $\mT$-monoid such that $([X,V],\conv)$ is a $\mT$-algebra. Then $([X,V],\conv^\ast)$ with $\ev_{X,V}:[X,V]\times X\to V$ form a $((-)\xtimes X)$-couniversal arrow for $V$.
\end{prop}

\begin{proof}
First note that $([X,V],\conv^\ast)$ is indeed a $\mT$-monoid by Proposition \ref{prop:T-AlgToT-Mon}, and that $\conv\cdot\conv^\ast=1_{[X,V]}$ (since $1_{[X,V]}=\conv\cdot\eta_{[X,V]}\le\conv\cdot\conv^\ast\le1_{[X,V]}$, see \cite{Sea:10}). Hence, for a given $f\in[X,V]$ the infimum of those $g\in[X,V]$ such that
\[
q\cdot T\ev_{X,V}\cdot\kappa_{[X,V],X}(\conv^\ast(f),\alpha(x))\le\ev_{X,V}(g,x)
\]
(for all $x\in X$) is $\conv\cdot\conv^\ast(f)=f$. As the order in $[X,V]$ is pointwise, one has
\[
q\cdot T\ev_{X,V}\cdot\kappa_{[X,V],X}(\conv^\ast(f),\alpha(x))\le\ev_{X,V}(f,x)\ ,
\] 
so that $\ev_{X,V}:[X,V]\xtimes X\to V$ is a $\mT$-monoid homomorphism. 

By Lemma~\ref{lem:ProductToOneVariable}, for a $\mT$-monoid $(Z,\gamma)$ and any homomorphism $g:Z\xtimes X\to V$, there is a unique map $f:Z\to[X,V]$ such that $\ev_{X,V}\cdot(f\times 1_X)=g$. We are therefore left to show that $f$ is a homomorphism. For this, note that
\begin{align*}
T\ev_{X,V}\cdot\kappa_{[X,V],X}\cdot(Tf\cdot\gamma\times\alpha)&=T\ev_{X,V}\cdot T(f\times 1_X)\cdot\kappa_{Z,X}\cdot(\gamma\times\alpha)\\
&=Tg\cdot\kappa_{Z,X}\cdot(\gamma\times\alpha)\\
&\le q^\ast\cdot g=q^\ast\cdot\ev_{X,V}\cdot(f\times 1_X)\ .
\end{align*}
By definition of $\conv$, we therefore have $\conv\cdot Tf\cdot\gamma\le f$, that is, $Tf\cdot\gamma\le\conv^\ast\cdot f$.
\end{proof}

\begin{thm}\label{thm:2}
Let $\mT$ be a lax monoidal powerset-enriched monad, $\sV$ a family of $\mT$-algebras that forms an initially dense family in $\Mons{\mT}$, and $(X,\alpha)$ a $\mT$-monoid. If $([X,V],\conv)$ is a $\mT$-algebra for all $(V,q)\in\sV$, then $(X,\alpha)$ is exponentiable in $\Mons{\mT}$.
\end{thm}

\begin{proof}
If $([X,V],\conv)$ is a $\mT$-algebra, we can apply Proposition \ref{prop:thm:2} for each $(V,q)\in\sV$. Thus, $(X,\alpha)$ is $\sV$-exponentiable, and therefore exponentiable by Proposition~\ref{prop:RestrictionToV}.
\end{proof}

\begin{rem}
It was shown in see \cite{Sea:10} that for a powerset-enriched monad $\mT$ the family of free $\mT$-algebras $\sV=(TX,\mu_X)_{X\in\ob\Set}$ is always an initially dense in $\Mons{\mT}$. Not only is this family $\sV$ rather large, but a free object can be difficult to exploit; hence, one would hope to find in $\Mons{\mT}$ a tractable initially dense family of $\mT$-algebras to play the role of a suitable $\sV$. In our examples, it turns out that the single object $(T1,\mu_1)$ alone can be used as test object.
\end{rem}

\begin{examples}\label{ex:ExpImpliesAlg}
\item\label{ex:ExpImpliesAlg:P} Recall from \ref{EMAlgsAndSup} that a $\mP$-algebra is a complete sup-semilattice. We choose $(V,q)$ to be the set $\two=\{0,1\}$ with structure map $q=\bigvee:P\two\to\two$; one has $q^\ast(x)=\dc\, x$ for all $x\in\two$, and $(\two,\dc_\two)$ is an initially dense object in $\PrOrd\cong\Mons{\mP}$. For a preordered set $(X,\dc_X)$, a monotone map $g\in[X,\two]$ can be identified with an up-closed subset $A_g\subseteq X$ (with $(x\in A_g\ \&\ x\le y)\implies y\in A_g$), the evaluation $\ev_{X,\two}$ with the truth-value of elementhood:
\[
\ev_{X,V}(g,x)=1\iff x\in A_g\ ,
\]
and the induced order on $[X,2]$ with inclusion. Denoting by $\Up X$ the set of up-closed subsets of $X$, the map $\conv:P[X,\two]\to[X,\two]$ of \ref{ssec:conv} is given by
\begin{align*}
\conv(\sA)&=\bigwedge\{g\in[X,\two]\mid\forall x\in X\ (P\ev_{X,\two}(\sA,\dc_X x)\subseteq\dc\, g(x))\}\\
&=\bigcap\{A\in\Up X\mid\forall x\notin A, \forall B\in\sA\ (x\notin B)\}\\
&=\bigcap\{A\in\Up X\mid{\textstyle\bigcup\sA\subseteq A}\}={\textstyle\bigcup\sA}
\end{align*}
for all $\sA\subseteq[X,\two]=\Up X$ (and with respect to the lax monoidal structure $\pi$ of \ref{ex:LaxMonoidalMonad}\ref{ex:LaxMonoidalMonad:P}). Since $\Up X$ (with union as supremum) is always a complete sup-semilattice, Theorem~\ref{thm:2} states that every preordered set is exponentiable (with respect to the cartesian product). This is just the well-known result that $\PrOrd$ is cartesian closed.
%%%%%%%%%%%
\item\label{ex:ExpImpliesAlg:F} For the filter monad $\mF$, we consider for $(V,q)$ the Sierpinski space $\two=\{0,1\}$ with structure map $q(\,\ff)=\bigwedge_{A\in\ff}\bigvee A$ for all $\ff\in F\two$. The neighborhood map $q^\ast:\two\to F\two$ is given by $q^\ast(0)=\{\{0\},\{0,1\}\}$ and $q^\ast(1)=\{0,1\}$, that is, the singleton $\{0\}$ is the only non-trivial open subset; moreover, $(\two,q^\ast)$ is initially dense in $\Top\cong\Mons{\mF}$. For a topological space $(X,\alpha)$, the set of continuous maps $[X,\two]$ can be identified with closed subsets of $X$, ordered by inclusion. Considering the product filter monoidal structure, one observes
\begin{align*}
\conv(\fF)&=\bigwedge\{g\in[X,\two]\mid\forall x\in X\ (F\ev_{X,\two}(\fF\times\alpha(x))\supseteq q^\ast\cdot g(x))\}\\
&=\bigcap\{A\subseteq X,\text{ closed}\mid\forall x\notin A, \exists\sA\in\fF, \exists N\in\alpha(x)\ (N\subseteq{\textstyle\bigcap_{B\in\sA}B^\C})\}\\
%&=\bigcap\{A\subseteq X,\text{ closed}\mid x\in A^\C\implies{\textstyle\bigcup_{\sA\in\fF}\bigcap_{B\in\sA}B^\C}\in\alpha(x)\}\\
&=\bigcap\{A\subseteq X,\text{ closed}\mid\forall x\in A^\C, \exists\sA\in\fF\ (\,\textstyle{\bigcap_{B\in\sA}B^\C}\in\alpha(x))\} 
\end{align*}
for all $\fF\in F[X,\two]$ (where $B^\C$ denotes the set-complement of $B\subseteq X$). The smallest closed subset of $X$ whose complement points each have one of the sets $\bigcap_{B\in\sA}B^\C$ ($\sA\in\fF$)  as neighborhood is
\[
\textstyle(\,\bigcup_{\sA\in\fF}(\,\bigcap_{B\in\sA}B^\C)^\circ)^\C=\bigcap_{\sA\in\fF}\overline{\bigcup_{B\in\sA}B}
\]
(using the notations $\overline{S}$ and $S^\circ$ for the respective closure and interior of $S\subseteq X$). Hence, for all $\fF\in F[X,\two]$,
\[
\conv(\fF)={\textstyle\bigwedge_{\sA\in\fF}\bigvee\sA}\ .
\]
The isomorphism between the category of $\mF$-algebras and the category of continuous lattices
\[
\Set^\mF\cong\Cnt
\]
(see for example \cite[Section 5.4]{Sea:11}) yields that $\conv$ is an $\mF$-algebra structure precisely when $[X,\two]^\op$ is a continuous lattice (in the sense of \cite{Gie/al:03}). Theorem \ref{thm:3} then returns the classical result that a topological space is exponentiable (with respect to the cartesian product) if and only if its set of open subsets, ordered by inclusion, is a continuous lattice\footnote{As is often the case when working in an abstract setting, the order that emerges naturally from existing structures clashes with the one usually appearing in the literature. To cite the ``classical result'' on exponentiability of topological spaces, we need the corresponding definition of a continuous lattice---that is opposite to the one presented in \cite[Section 5.4]{Sea:11}. Reversing the order on $[X,\two]$ then leads us to switch from closed to open subsets.}.
\end{examples}

%\subsection{The main Theorem}
%With Theorems \ref{thm:1} and \ref{thm:2}, we obtain the following characterization of exponentiable $\mT$-monoids.
%
%\begin{thm}\label{thm:3}
%Let $\mT$ be a lax monoidal powerset-enriched monad, and $\sV$ a family of $\mT$-algebras that forms an initially dense family in $\Mons{\mT}$. A $\mT$-monoid $(X,\alpha)$ is exponentiable if and only if $([X,V],\conv)$ is a $\mT$-algebra for all $(V,q)\in\sV$.
%\end{thm}
%
%\begin{proof}
%The statement summarizes Theorems \ref{thm:1} and \ref{thm:2}.
%\end{proof}
%
%\begin{cor}\label{cor:thm:3}
%Let $\mT$ be a lax monoidal powerset-enriched monad, and $(V,q)$ a $\mT$-algebra such that $(V,q^\ast)$ is an initially dense object in $\Mons{\mT}$. A $\mT$-monoid $(X,\alpha)$ is exponentiable if and only if $([X,V],\conv)$ is a $\mT$-algebra.
%\end{cor}
%
%\begin{proof}
%This is the $\sV=\{(V,q)\}$ case of Theorem \ref{thm:3}.
%\end{proof}

\section{Applications}\label{sec:Examples}%%%%%%%%%%%%%%%%%%%%%%%%%%%%%%%%%

\subsection{Interior spaces}\label{Int}
The \df{up-set monad} $\mU=(U,\mu,\eta)$ arises from the adjunction
\[
\Set(-,\two)\dashv\PrOrd(-,\two):\PrOrd^\op\to\Set\ .
\]
By identifying $\Set(X,\two)$ with the powerset $PX$, we observe that $UX=\Up PX$ is the set of up-closed subsets in $PX$, and this monad can then be described as follows:
\[
B\in U f(\fx)\iff f^\inv(B)\in\fx\ ,\quad A\in\eta_X(x)\iff x\in A\ ,\quad A\in\mu_X(\fX)\iff A^\mU\in\fX\ ,
\]
for $f:X\to Y$, $\fx\in UX$, $x\in X$, $\fX\in UUX$, and where $A^\mU=\{\fx\in UX\mid A\in\fx\}$. The  monad morphism $\tau:\mP\to\mU$ that associates to $A\in PX$ the principal filter
\[
\tau_X(A):=\{B\in PX\mid A\subseteq B\}
\]
makes $\mU$ powerset-enriched. As in the filter case, the order induced by $\tau$ on the sets $UX$ (via the identification $\Set^\mP\cong\Sup$ as in \ref{EMAlgsAndSup}) is given by reverse subset-inclusion:
\[
\fx\le\fy\iff\fx\supseteq\fy
\]
for all $\fx,\fy\in UX$, and one has
\[
\Set^\mU\cong\Ccd\dand \Mons{\mU}\cong\Int\ ,
\]
where $\Ccd$ and $\Int$ denote the category of constructively completely distributive lattices and the category of interior spaces, respectively (see \cite{Sea:09}).

Let us note that the exponentiable spaces with respect to the cartesian product in $\Int$ are the indiscrete ones (see \cite[Corollary 2.3]{ClaColSon:01}). There is nevertheless a family of maps $\kappa_{X,Y}:UX\times UY\to U(X\times Y)$ given by
\[
\kappa_{X,Y}(\fx,\fy):=\{A\times B\mid A\in\fx, B\in\fy\}
\]
(for all $\fx\in UX$, $\fy\in UY$) that makes $(\mU,\tau)$ lax monoidal, and allows us to define a monoidal structure $(-)\xtimes(-)$ on $\Int$. The open subsets of the box product $(X,\alpha)\xtimes(Y,\beta)$ of two interior spaces are given by arbitrary unions of open subsets of the form $V\times W$. By identifying open subsets of $(X,\alpha)$ with the function space $[X,\two]^\op$ and observing that the Sierpinski space $\two$ is also initially dense in $\Int$, Theorem \ref{thm:3} states that the exponentiable interior spaces are those $X$ whose set of open subsets forms a constructively completely distributive lattice. %Thus, by replacing the cartesian product by the box product, we obtain exponentiable interior spaces that closely mirrors the topological case.%This result  faithfully reproduces the topological case, and provides an interesting collection of exponentiable spaces. 

\subsection{Quantale-enriched categories}
Given a quantale $V$ with tensor $\otimes$ and neutral element $k$, the \df{$V$-powerset functor} $P_V$ sends a set $X$ to its \df{$V$-powerset} $V^X$, and a function $f:X\to Y$ to $P_V f:V^X\to V^Y$, where
\[
P_V f(\phi)(y):=\textstyle{\bigvee_{x\in f^\inv(y)}\phi(x)}\ ,
\]
for all $\phi\in V^Y$, $y\in Y$. The multiplication $\mu:P_V P_V\to P_V$ and unit $\eta:1_\Set\to P_V$ of the \df{$V$-powerset monad} $\mP_V$ on $\Set$ are given respectively by
\[
\mu_X(\Phi)(y):=\textstyle{\bigvee_{\phi\in V^X}
\Phi(\phi)\otimes\phi(y)}
\dand
\eta_X(x)(y):=
\begin{cases}
k&\text{ if $x=y$}\\
\bot&\text{ otherwise\ ,}
\end{cases}
\]
for all $x,y\in X$, $\Phi\in V^{V^X}$ (where $\bot$ denotes the bottom element of the lattice). There is a  monad morphism $\tau:\mP\to\mP_V$ whose components $\tau_X$ send a subset $A\subseteq X$, identified with its characteristic function $\chi_A:X\to\{\bot,\top\}$, to the characteristic function $\tau_X(A)=\iota\cdot\chi_A$ of $A$ in $V^X$:
\[
\tau_X(A)(x):=\begin{cases}
k&\text{ if $x\in A$}\\
\bot&\text{ otherwise.}
\end{cases}
\]
With this monad morphism $\mP_V$ becomes a lax monoidal powerset-enriched monad and one has
\[
\Set^{\mP_V}\cong\Sup^V\dand\Mons{\mP_V}\cong\Cats{V}\ ,
\]
where $\Sup^V$ is the category of $V$-actions in $\Sup$, and $\Cats{V}$ is the category of small categories enriched in $V$ (see \cite{Stu:06}). The order induced by $\tau$ on the sets $V^X$ is the pointwise order, a $\mP_V$-algebra $(X,a)$ is a complete sup-semilattice with structure map $a:V^X\to X$ corresponding to an action $a(\phi)=\bigvee_{x\in X}\phi(x)\otimes x$, and the $V$-category $(V,q^\ast)$ is initially dense in $\Cats{V}$.

\subsubsection{The monoidal structure} If the quantale $V$ is commutative, then for $\phi\in P_V X$ and $\psi\in P_V Y$,
\[
\kappa_{X,Y}(\phi,\psi)=\phi\otimes\psi
\]
defines a natural transformation $\kappa:P_V(-)\times P_V(-)\to P_V(-\times-)$ that makes $(\mP_V,\tau)$ lax monoidal. In this case, one computes for a $V$-category $(X,\alpha)$ and $\phi\in V^{[X,V]}$,
\begin{align*}
\conv(\phi)&=\bigwedge\{g\in[X,V]\mid\forall x\in X\  (q\cdot P_V\ev_{X,V}(\phi\otimes\alpha(x))\le g(x)\}\\
&=\bigwedge\{g\in[X,V]\mid\forall x\in X\  ({\textstyle\bigvee_{v\in V,(f,y)\in\ev^{-1}(v)}\phi(f)\otimes\alpha(x)(y)\otimes v}\le g(x))\}\\
&={\textstyle \bigvee_{f\in[X,V]}\phi(f)\otimes f}\ ,
\end{align*}
since $k\le \alpha(x)(x)$, and $f\in[X,V]$ means that $\bigvee_{y\in X}\alpha(x)(y)\otimes f(y)\le f(x)$ for all $x\in X$. Mirroring the powerset case, every hom-set $[X,V]$ of $\Cats{V}$ is a $V$-action with this structure, so that Theorem~\ref{thm:2} simply states the well-known fact that $\Cats{V}$ is monoidal closed (with respect to its natural monoidal structure induced by the tensor of $V$ as above).

\subsubsection{The cartesian structure} If the underlying lattice of $V$ is a frame and $k=\top$, the maps $\rho_{X,Y}:P_V X\times P_V Y\to P_V(X\times Y)$ given by
\[
\rho_{X,Y}(\phi,\psi)=\phi\wedge\psi
\]
(for all $\phi\in P_V X$, $\psi\in P_V Y$) form a natural transformation that defines the cartesian product in $\Cats{V}$. This natural transformation makes $(\mP_V,\tau,\rho)$ into a lax monoidal monad. 

For a subset $\sA\subseteq[X,V]$, one computes
\begin{align*}
\conv(\chi_\sA)&=\bigwedge\{g\in[X,V]\mid\forall x\in X\  (q\cdot P_V\ev_{X,V}(\chi_\sA\wedge\alpha(x))\le g(x))\}\\
&=\bigwedge\{g\in[X,V]\mid\forall x\in X\  ({\textstyle\bigvee_{f\in[X,V],y\in X}(\chi_\sA(f)\wedge\alpha(x)(y))\otimes f(y)}\le g(x))\}\\
&=\bigwedge\{g\in[X,V]\mid\forall x\in X\  ({\textstyle\bigvee_{f\in\sA,y\in X}(k\wedge\alpha(x)(y))\otimes f(y)}\le g(x))\}\ .
\end{align*}
Since $k\le\alpha(x)(x)$ and $\bigvee_{y\in X}\alpha(x)(y)\otimes f(y)\le f(x)$, one has
\[
\conv(\chi_\sA)=\textstyle{\bigvee\sA}
\]
for all $\sA\subseteq[X,V]$; hence, if $\conv:P_V[X,V]\to[X,V]$ is a $\mP_V$-algebra structure, the order induced on $[X,V]$ by $\conv\cdot\tau_{[X,V]}$ is the pointwise order. One also has $\conv\cdot \eta_{[X,V]}=1_{[X,V]}$ since $\eta_{[X,V]}(f)=\chi_{\{f\}}$ for all $f\in[X,V]$.

With $k=\top$ in $V$, on has for $\psi\in P_V[X,V]$ that the map $\tilde{\psi}:X\to V$, defined by
\[
\tilde{\psi}(y):=\textstyle\bigvee_{f\in[X,V],z\in X}(\psi(f)\wedge\alpha(y)(z))\otimes f(z)\ ,
\]
is a $\mP_V$-monoid homomorphism. Indeed, for $x\in X$, one has
\begin{align*}
q\cdot P_V\tilde{\psi}(\alpha(x))&\textstyle=\bigvee_{y\in X}\alpha(x)(y)\otimes\tilde{\psi}(y)\\
&\textstyle=\bigvee_{y\in X}\alpha(x)(y)\otimes\big(\bigvee_{f\in[X,V],z\in X}(\psi(f)\wedge\alpha(y)(z))\otimes f(z)\big)\\
&\textstyle\le \bigvee_{f\in[X,V],z\in X}\big( (\,\bigvee_{y\in X}\alpha(x)(y)\otimes\psi(f) )\wedge(\,\bigvee_{y\in X}\alpha(x)(y)\otimes\alpha(y)(z))\big)\otimes f(z)\\
&\textstyle\le \bigvee_{f\in[X,V],z\in X}(\psi(f)\wedge\alpha(x)(z)))\otimes f(z)=\tilde{\psi}(x)\ ,
\end{align*}
so that $P_V\tilde{\psi}\cdot\alpha\le q^\ast\cdot\tilde{\psi}$. Hence, for any $\psi\in P_V[X,V]$,
\[
\conv(\psi)=\tilde{\psi}\ ,
\]
The condition $\conv\cdot P_V\conv(\Phi)=\conv\cdot\mu_{[X,V]}(\Phi)$ for all $\Phi\in V^{V^[X,V]}$ then becomes
\begin{equation}\label{cond:ExpV-CatCartesianPrelim}\tag{$\dagger$}
\textstyle\bigvee(\Phi(\psi)\wedge\alpha(x)(y))\otimes(\psi(f)\wedge\alpha(y)(z))\otimes f(z)=\bigvee\big((\Phi(\psi)\otimes\psi(f))\wedge\alpha(x)(z)\big)\otimes f(z)
\end{equation}
for all $\Phi\in V^{V^{[X,V]}}$ and $x\in X$, with the supremum on each side ranging over all $\psi\in V^{[X,V]}$, $f\in[X,V]$, and $y,z\in X$. As the left-hand side is always smaller than the right-hand side, this condition can be read with a ``$\ge$'' sign instead of equality. For a chosen $t\in X$, if $\psi_{v,t}\in V^{[X,V]}$ is the map that sends $\alpha(-)(t)\in[X,V]$ to $v\in V$ and all other $f\in[X,V]$ to the bottom element $\bot\in V$, and $\Phi_{u,v,t}\in V^{V^{[X,V]}}$ is the map that sends $\psi_{v,t}$ to $u\in V$ and all other maps to $\bot$, one obtains from condition \eqref{cond:ExpV-CatCartesianPrelim} and $k=\top$:\begin{align*}
\textstyle\bigvee_{y\in X}(u\wedge\alpha(x)(y))\otimes(v\wedge\alpha(y)(t))&\textstyle\ge\bigvee_{z\in X}\big((u\otimes v)\wedge\alpha(x)(z)\big)\otimes\alpha(z)(t)\\
&\textstyle\ge (u\otimes v)\wedge\alpha(x)(t)\ .
\end{align*}
Hence, \eqref{cond:ExpV-CatCartesianPrelim} implies
\[
\textstyle\bigvee_{y\in X}(u\wedge\alpha(x)(y))\otimes(v\wedge\alpha(y)(z))\ge(u\otimes v)\wedge\alpha(x)(z)
\]
for all $u,v\in V$ and $x,z\in X$. Since \eqref{cond:ExpV-CatCartesianPrelim} is also a consequence of this last condition, and $\conv\cdot\eta_{[X,V]}=1_{[X,V]}$ always holds, Theorem \ref{thm:3} recovers the characterization from \cite[Corollary 3.5]{CleHof:06} of exponentiable $V$-categories $(X,\alpha)$ (with respect to the cartesian product)  when the underlying lattice of $V$ is a frame and $k=\top$. Note that commutativity of the tensor of $V$ is not necessary.

\bibliographystyle{plain}

\end{document}